\documentclass[11pt, a4paper,leqno]{amsart}
\usepackage{amsmath,amsthm,amscd,amssymb,amsfonts, amsbsy}
\usepackage{latexsym}
\usepackage{txfonts}
\usepackage{exscale}

\usepackage{pgf}

\calclayout
\allowdisplaybreaks

\theoremstyle{plain}
\newtheorem{theorem}[equation]{Theorem}

\theoremstyle{definition}

\theoremstyle{remark}

\numberwithin{equation}{section}

\newcommand{\R}{{\mathbb R}}

\def\mean#1{\mathchoice%
          {\mathop{\kern 0.2em\vrule width 0.6em height 0.69678ex depth -0.58065ex
                  \kern -0.8em \intop}\nolimits_{\kern -0.4em#1}}%
          {\mathop{\kern 0.1em\vrule width 0.5em height 0.69678ex depth -0.60387ex
                  \kern -0.6em \intop}\nolimits_{#1}}%
          {\mathop{\kern 0.1em\vrule width 0.5em height 0.69678ex
              depth -0.60387ex
                  \kern -0.6em \intop}\nolimits_{#1}}%
          {\mathop{\kern 0.1em\vrule width 0.5em height 0.69678ex depth -0.60387ex
                  \kern -0.6em \intop}\nolimits_{#1}}}

\DeclareMathOperator{\diam}{diam}

\def\div{\mathop{\operatorname{div}}\nolimits}

\begin{document}
\title[Extension properties and boundary estimates for a fractional heat operator]{Extension properties and boundary estimates\\ for a fractional heat operator}

\address{Kaj Nystr\"{o}m\\Department of Mathematics, Uppsala University\\
S-751 06 Uppsala, Sweden}
\email{kaj.nystrom@math.uu.se}

\address{Olow Sande\\Department of Mathematics, Uppsala University\\
S-751 06 Uppsala, Sweden}
\email{olow.sande@math.uu.se}

\author{K. Nystr{\"o}m, O. Sande}

\maketitle
\begin{abstract}
\noindent
The square root of the heat operator $\sqrt{\partial_t-\Delta}$, can be realized as the Dirichlet to Neumann map of the heat extension of data on
 $\mathbb R^{n+1}$ to $\mathbb R^{n+2}_+$. In this note we obtain similar characterizations for general fractional powers of the heat operator,
 $(\partial_t-\Delta)^s$, $s\in (0,1)$. Using the characterizations we derive properties and boundary estimates for parabolic integro-differential equations from purely local arguments in the extension problem.
 \\

\noindent
2000  {\em Mathematics Subject Classification.}
\noindent

\medskip

\end{abstract}

    \section{Introduction}
    In recent years there has been a surge in the study of the fractional Laplacian $(-\Delta)^s$ as well as more general linear and non-linear fractional operators. From an applied perspective a natural parabolic extension of $(-\Delta)^s$ is the parabolic operator $\partial_t+(-\Delta)^s$ which appears, for example, in the study of stable processes and in option pricing models, see \cite{CFi} an the references therein. An other generalization is the time-fractional diffusion equation $\partial_t^\beta+(-\Delta)^s$ being the sum of a fractional and non-local time-derivative as well as a non-local operator in space as well. This type of equations has attracted considerable  interest during the last years, mostly due to their applications in the modeling
of anomalous diffusion, see \cite{ACV}, \cite{KSVZ}, \cite{KSZ}, and  the references therein. Decisive progress in the study of the fine properties of solutions to $(-\Delta)^su=0$ has been achieved through an extension technique, rediscovered in \cite{CSi}, based on which the fractional Laplacian can be study through a local but degenerate elliptic operator having degeneracy determined by an $A_2$-weight. The latter operators have been thoroughly studied in \cite{FKS}, \cite{FJK}, \cite{FJK1}, as well as in several other subsequent papers.
Due to the lack of an established extension technique for operators of the forms
    $\partial_t+(-\Delta)^s$, $\partial_t^\beta+(-\Delta)^s$, more modest, but still important, progress has been made concerning these equations, again see
     \cite{CFi}, \cite{ACV}, \cite{KSVZ}, \cite{KSZ}, and the references therein.

In this note we take a different approach by considering directly the fractional heat operator  $(\partial_t-\Delta)^s$. Given $s\in (0,1)$ we introduce the fractional heat operator $(\partial_t-\Delta)^s$ defined on the Fourier transform side by  multiplication with the multiplier
\[
(|\xi|^2-i\tau)^s.
\]
Using  \cite{Sa} it follows that $(\partial_t-\Delta)^s$ can be realized as a parabolic hypersingular integral,
     \begin{eqnarray*}\label{aa}
     (\partial_t-\Delta)^sf(x,t)=\frac {1}{\Gamma(-s)}\int_{-\infty}^t\int_{\mathbb R^n}\frac {(f(x,t)-f(x',y'))}{(t-t')^{1+s}}W(x-x',t-t')\, dx'dt',
     \end{eqnarray*}
where $W(x,t)=(4\pi t)^{-n/2}\exp(-|x|^2/(4t))$
for $t>0$
and where $\Gamma(-s)$ is the gamma function evaluated at $-s$. The main result established in this note is
that, in analogy with \cite{CSi}, fine properties of solutions to  $(\partial_t-\Delta)^sf=0$ can be derived through an extension technique based on which the fractional heat operator can be studied through a local but degenerate parabolic operator having degeneracy determined by an $A_2$-weight. To be precise, we consider a specific extension to the upper half space
\[
\mathbb R^{n+2}_+=\{(X,t)=(x,x_{n+1},t)\in \mathbb R^{n}\times\mathbb R\times\mathbb R:\ x_{n+1}>0\},
\]
having boundary
\[
\mathbb R^{n+1}=\{(x,x_{n+1},t)\in \mathbb R^{n}\times\mathbb R\times\mathbb R:\ x_{n+1}=0\}.
\]
In the following we let $\nabla =(\nabla_x,\partial_{x_{n+1}})$ an we let $\div$ be the associated divergence operator. Let $a=1-2s$. Letting
\begin{eqnarray}\label{aah}\Gamma_{x_{n+1}}(x,t):=\frac {1}{4^s\Gamma(-s)}x_{n+1}^{1-a}\frac {1}{t^{1+s}}W(x,t)\exp(-|x_{n+1}|^2/(4t))
\end{eqnarray}
whenever $(x,x_{n+1},t)\in \mathbb R^{n+2}_+$ and $t>0$,
we introduce, given $a$ and $f\in C_0^\infty(\mathbb R^{n+1})$, the function
 \begin{eqnarray}\label{aa+}
 u(X,t)=u(x,x_{n+1},t)=\int_{-\infty}^t\int_{\mathbb R^n}f(x',t')\Gamma_{x_{n+1}}(x-x',t-t')\, dx'dt'.
     \end{eqnarray}
     Given $ (x,t) \in \mathbb R^{n+1}$ and $r >0$,
let $B(x,r)$ denote the standard Euclidean ball and let $C_r(x,t)$ denote the standard parabolic cylinder
\[
C_r (x,t)=B(x,r)\times (t-r^2, t+r^2).
\]
Our first result is the following theorem.
     \begin{theorem}\label{thm1} Consider $s$, $0<s<1$, fixed and let $a=1-2s$. Consider $f\in C_0^\infty(\mathbb R^{n+1})$ and let
      $u$ be defined as in \eqref{aa+}. Then $u$ solves
      \begin{eqnarray}\label{aa++b+a}
            {x_{n+1}}^a\partial_tu(X,t)-\div({x_{n+1}}^a\nabla u(X,t))&=&0,\quad\quad (X,t)\in\mathbb R^{n+2}_+,\notag\\
            u(x,0,t)&=&f(x,t),\ (x,t)\in\mathbb R^{n+1},
\end{eqnarray}
and
      \begin{eqnarray*}\label{aa++q+}
      {x_{n+1}}^a\partial_{x_{n+1}} u(X,t)\biggl |_{x_{n+1}=0}=-\lim_{{x_{n+1}}\to 0} 4^{s}\frac {u(X,t)-u(x,0,t)}{{x_{n+1}}^{1-a}}=(\partial_t-\Delta)^sf(x,t).
      \end{eqnarray*}
      Furthermore, assume that $(\partial_t-\Delta)^sf(x,t)=0$ whenever $(x,t)\in C_r(\tilde x,\tilde t)$, for some $(\tilde x,\tilde t)\in\mathbb R^{n+1}$, $r>0$, let $\tilde u(x,x_{n+1},t)$ be defined to equal
      $u(x,x_{n+1},t)$ whenever $x_{n+1}\geq 0$ and defined to equal
      $u(x,-x_{n+1},t)$ whenever $x_{n+1}<0$. Then $\tilde u$ is a weak solution to the equation
      \begin{eqnarray*}\label{aa++b+auu}
            {|x_{n+1}|}^a\partial_t\tilde u(X,t)-\div({|x_{n+1}|}^a\nabla \tilde u(X,t))=0,
\end{eqnarray*}
in $\{(X,t)=(x,x_{n+1},t)\in\mathbb R^{n+2}: (x,t)\in C_r(\tilde x,\tilde t),\ x_{n+1}\in (-1,1)\}$.
\end{theorem}

      Note that Theorem \ref{thm1} is in line with the fact that $\sqrt{\partial_t-\Delta}f$ can be realized as the Dirichlet to Neumann map of the heat extension of $f$ to $\mathbb R^{n+2}_+$. We also note that a nice feature of Theorem \ref{thm1}   is that if $f$ is independent of $t$ then all of the above objects coincide with the corresponding objects appearing in the study of the fractional Laplacian through the extension technique considered in \cite{CSi}.

      Based on Theorem \ref{thm1} we can derive refined estimates for non-negative solutions to
      $(\partial_t-\Delta)^sf(x,t)=0$ in a domain by establishing the corresponding estimates for non-negative solutions to the equation
            \begin{eqnarray}\label{aa++b+auu1}
            {|x_{n+1}|}^a\partial_t\tilde u(X,t)-\div({|x_{n+1}|}^a\nabla \tilde u(X,t))=0,
\end{eqnarray}in a natural extended domain. Let in the following $\Omega_T=\Omega\times (0,T)$, $T>0$, where $\Omega\subset\mathbb R^n$ is a bounded domain, i.e., a bounded, connected and open set in $\mathbb R^n$.
Let the parabolic boundary of the cylinder $\Omega_T$, $\partial_p\Omega_T$, be defined as
\[
\partial_p\Omega_T = S_T \cup (\bar \Omega \times \{0\}),\ \ \ S_T = \partial \Omega \times [0,T).
\]
Recall that  $\Omega\subset\mathbb R^{n}$ is a bounded Lipschitz domain if there exists a finite set
of balls $\{B(\hat x_i,r_i)\}$,
with $\hat x_i\in \partial\Omega$ and $r_i>0$,
such that $\{B(\hat x_i,r_i)\}$ constitutes a covering of an open neighborhood of
$\partial\Omega$ and such that, for each $i$,
\begin{eqnarray}\label{Lip}
\Omega\cap B(\hat x_i,r_i)&=&\{y=(y',y_n)\in\mathbb R^{n} : y_n > \phi_i ( y')\}\cap B(\hat x_i,r_i),\notag
\\
\partial\Omega\cap B(\hat x_i,r_i)&=&\{y=(y',y_n)\in\mathbb R^{n} : y_n= \phi_i ( y')\}\cap B(\hat x_i,r_i),
\end{eqnarray}
in an appropriate coordinate system and for a Lipschitz function $\phi_i : \mathbb R^{n-1} \to \mathbb R.
$ The Lipschitz constants of $\Omega$ are defined to be $M=\max_i\| \,|\nabla\phi_i| \, \|_\infty$,
$r_0:=\min_i r_i$ and we will often refer to $\Omega$ as a Lipschitz domain  with parameters $M$ and $r_0$.
If $\Omega$ is a Lipschitz domain  with parameters $M$ and $r_0$, then there exists,
for any $\hat x\in\partial\Omega$,
$0<r<r_0$, a point $A_r(\hat x)\in\Omega$, such that
\[
M^{-1}r<d(\hat x,A_r(\hat x))<r,\ \ \text{and}\ \
d(A_r(\hat x),\partial\Omega)\geq M^{-1}r.
\]

We let $\diam(\Omega) = \sup\{|x-y|\mid x,y\in\Omega\}$ denote the Euclidean diameter of $\Omega$.
When we in the following write that a constant $c$ depends on the operator $\mathcal{H}$, $c=c(\mathcal{H})$,  we mean that $c$ depends on the dimension $n$, and $s$. Our second result is the following theorem.

\begin{theorem}\label{T:quotients}
Let $\Omega\subset\mathbb R^n$ be a bounded
Lipschitz domain
with parameters $M$, $r_0$
and let $\Omega_T=\Omega\times (0,T)$ for some $T>0$.
Let $f_1, f_2$ be non-negative
solutions of $(\partial_t-\Delta)^{s}f=0$ in $\Omega_T$
vanishing continuously on $S_T$. Let $\delta$, $0<\delta<r_0$, be a fixed constant.
Then $f_1/f_2$ is H{\"o}lder continuous
on the closure of $\Omega\times(\delta^2,T]$.
Furthermore, let $(x_0,t_0)\in S_T$, $\delta^2\leq t_0$, and assume that $r<\delta/2$.
Then there exist  $c=c(n,s,M,\diam(\Omega),T,\delta)$, $1\leq c<\infty$,
and $\alpha=\alpha(n,s,M,\diam(\Omega),T,\delta)$, $\alpha\in (0,1)$, such that
\[
\biggl |\frac { f_1 ( x,t ) }{ f_2 ( x,t ) }-\frac { f_1 ( y,s ) }{ f_2 ( y,s ) }\biggr |
\leq c\biggl ( \frac{|x-y|+|s-t|^{1/2}}{r} \biggr)^\alpha
\frac {f_1\bigl( A_r(x_0),t_0\bigr) }{ f_2 \bigl(A_r(x_0),t_0\bigr) }
\]
whenever $(x,t), (y,s)\in \Omega_T\cap C_{r/4}(x_0,t_0)$.
\end{theorem}

The weight
$\lambda(x,x_{n+1})=|x_{n+1}|^a$ is easily seen to be an $A_2$-weight on $\mathbb R^{n+1}$ and hence our extension operator in \eqref{aa++b+auu1} can be embedded into a larger class of operators of the form
  \begin{equation}\label{lop}
\mathcal{H}=\lambda(X)\partial_t-\sum_{i,j=1}^{n+1}\partial_{x_i}(a_{ij}(X)\partial_{x_j}),
\end{equation}
in $\mathbb R^{n+2}$,  where $A(X)=\{a_{ij}(X)\}$ is measurable, real, symmetric and
\begin{equation}\label{lop+}
\beta^{-1}\lambda(X)|\xi|^2\leq \sum_{i,j=1}^{n+1}a_{ij}(X)\xi_i\xi_j\leq\beta\lambda(X)|\xi|^2\mbox{ for all }(X,t)\in\mathbb R^{n+2},\ \xi\in\mathbb R^{n+1},
\end{equation}
for some constant $\beta\geq 1$ and for some non-negative and real-valued function $\lambda=\lambda(X)$ such that
\begin{equation}\label{aabb}
\lambda\in A_2(\R^{n+1}) \mbox{ with } A_2 \mbox{-constant } \Lambda.
\end{equation}
Based on  Theorem \ref{thm1}, the proof of Theorem \ref{T:quotients} follows from the corresponding  boundary  estimates which we establish for non-negative solutions to $\mathcal{H}u=0$, assuming  \eqref{lop}- \eqref{aabb},  in domains of the form $\tilde\Omega_T=\tilde\Omega\times (0,T)$ where \begin{equation}\label{dom}\tilde\Omega=\{X=(x,x_{n+1})\in \mathbb R^{n+1}:\ x\in\Omega,\ x_{n+1}\in (a,b)\},
\end{equation}
for some $a<b$. In this case we let $\tilde S_T = \partial \tilde \Omega \times [0,T)$ denote the lateral boundary of $\tilde\Omega_T$.  As it turns out, there is a limited literature devoted to the type of operators defined by $\mathcal{H}$. Indeed, in
\cite{CS} the (standard) parabolic Harnack inequality is established, in \cite{CUR} Gaussian estimates were established, and in \cite{Is} an estimate previously establish in the context of uniformly parabolic equations by Salsa \cite{S}, was generalized to the operators considered in this note. For an extensive study of equations as in \eqref{lop}, allowing also for time-dependent coefficients, without the presence of $\lambda$  on the left hand side, assuming \eqref{lop+} for a real-valued function $\lambda=\lambda(X)$
belonging to the Muckenhoupt class $A_{1+2/(n+1)}(\mathbb R^{n+1})$, we refer to \cite{NPS}.

The organization of this note is as follows. In Section 2 we prove Theorem \ref{thm1}. Using Theorem \ref{thm1} it follows that to prove Theorem \ref{T:quotients} it suffices to prove Theorem \ref{T:quotients.old} which is stated and proved in Section 3. Theorem \ref{T:quotients.old} concerns non-negative solutions to $\mathcal{H}u=0$ where $\mathcal{H}$ is of the form \eqref{lop}, assuming \eqref{lop+} and \eqref{aabb}. Note that $A=\{a_{i,j}\}$ is independent of $t$ and it is then easily seen, using \cite{CS} and  \cite{CUR}, that Theorem \ref{T:quotients.old} follows by the now classical arguments in \cite{FGS} once the continuous Dirichlet problem can be solved and once H{\"o}lder continuity up to the lateral boundary can be established.

Finally we note that while completing this note two related papers have been  posted on math arxiv. First, in \cite{BFe} the authors pursue ideas similar to ours by considering an extension problem associated to a fractional time-derivative. The extension problem is then a problem set in $\mathbb R\times \mathbb R_+$ and the extension operator considered is identical to ours in the case when we only consider the operator $\partial_t^s$. However, our set-up is different to \cite{BFe} in the sense that we consider the operator $(\partial_t-\Delta)^s$, and as we establish Theorem \ref{thm1} and Theorem \ref{T:quotients}. Second, in \cite{ST}, submitted to the arxiv on the very same day as we were to submit this note, the authors independently consider exactly the same set-up as we do but their focus is slightly different compared to ours. In particular, they develop a regularity theory for solutions to
\[
(\partial_t-\Delta)^sf(x,t)=h.
\]
This is done by using the same extension problem as we do in this note to characterize the nonlocal equation with a local
degenerate parabolic equation. Using this in \cite{ST} the authors prove a parabolic Harnack
inequality and a version of the result of Salsa \cite{Sa} referred to above. Subsequently, H{\"o}lder and Schauder estimates for the space-time Poisson problem are deduced via a characterization of parabolic H{\"o}lder spaces. The conclusion is that apparently there is an emerging interest in the community to study the operator $(\partial_t-\Delta)^s$ and we believe that this note and the contributions in \cite{BFe}, \cite{ST} all complement each other. In particular, Theorem \ref{T:quotients} is unique to this paper.

\section{Proof of Theorem \ref{thm1}}
Let $\Omega\subset\mathbb R^{n}$ be a bounded domain and let $\tilde\Omega$ be defined as in \eqref{dom}. Let in the following $\lambda\in A_{2}(\mathbb R^{n+1})$. Let $ L^2_\lambda ( \tilde\Omega )$ denote the
Hilbert space of functions defined on $\tilde\Omega$ which are square integrable
on $\tilde\Omega$ with respect to the measure $\lambda(X)dX$.
Let $ L^2_\lambda ( \tilde\Omega )$ be equipped with the natural weighted $L^2$-norm
$ \| \cdot \|_{L_\lambda^2 ( \tilde\Omega )} $.
Furthermore, let $ W^{1 ,2}_\lambda ( \tilde\Omega )$ be the space of equivalence
classes of functions $ u $ with distributional gradient
$ \nabla u = ( u_{x_1}, \dots, u_{x_{n+1}} ), $ both of which belong to $L^2_\lambda(\tilde\Omega)$.
Let
\begin{equation}
\| u \|_{ W^{1,2}_\lambda (\tilde\Omega)}
=
\| u \|_{ L^2_\lambda  (\tilde\Omega)} + \|
\,
| \nabla u | \, \|_{ L^2_\lambda  ( \tilde\Omega )} \,
\end{equation}
be the norm in $ W^{1,2}_\lambda  ( \tilde\Omega ) $.
Let $ C^\infty_0 (\tilde\Omega )$ denote the
set of infinitely differentiable functions with compact support in $\tilde\Omega$
and let $ W^{1 ,2}_{\lambda,0} ( \tilde\Omega )$ denote the closure of
$ C^\infty_0 (\tilde\Omega )$ in the norm $\| \cdot\|_{ W^{1,2}_\lambda  (\tilde\Omega)}$.
$ W^{1,2}_{\lambda,{loc}} ( \tilde\Omega ) $ is defined in the standard way.
Given $t_1<t_2$, let $L^2(t_1,t_2,W^{1,2}_\lambda  ( \tilde\Omega ))$ denote the space of functions such that
for almost every $t$, $t_1\leq t\leq t_2$,
the function $x\to u(x,t)$ belongs to
$W^{1,2}_\lambda ( \tilde\Omega )$ and
\begin{equation}
\| u \|_{ L^2(t_1,t_2,W^{1,2}_\lambda  ( \tilde\Omega ))}:=\biggl (\int\limits_{t_1}^{t_2}\int\limits_{\tilde\Omega}\biggl (|u(X,t)|^2+|\nabla u(X,t)|^2\biggr )\lambda(X)dXdt\biggr )^{1/2} <\infty.
\end{equation}
The space
$L^2\left(t_1,t_2,W^{1,2}_{\lambda,loc}(\tilde\Omega)\right)$ is defined analogously.
Let $\mathcal{H}$ be as in  \eqref{lop}, assume \eqref{lop+} and \eqref{aabb}. A function $u$ is said to be a weak solution of $\mathcal{H}u=0$ in $ \tilde\Omega_T$ if,
for all open sets $\tilde\Omega'\subseteq \tilde\Omega$ and  $0<t_1<t_2<T$,
we have $  u \in L^2(t_1,t_2,W^{1,2}_\lambda ( \tilde\Omega' ))$  and
\begin{eqnarray}\label{1.1}
&&
\int\limits_{t_1}^{t_2}\int\limits_{\tilde\Omega'}a_{ij}(X)\partial_{x_i}u\partial_{x_j}\theta dXdt
-
\int\limits_{t_1}^{t_2}\int\limits_{\tilde\Omega'}u\partial_t\theta \lambda dXdt\notag
\\
&&
+
\int\limits_{\tilde\Omega'}u(X,t_2)\theta(x,t_2) \lambda dX
-
\int\limits_{\tilde\Omega'}u(X,t_1)\theta(x,t_1) \lambda dX
=0
\end{eqnarray}
whenever $ \theta  \in C_0^\infty(\tilde\Omega'_T)$.
Furthermore, $u$ is said to be a \emph{weak supersolution}
to $\mathcal{H}u=0$ if the left hand side of \eqref{1.1} is non-negative for all
$\theta \in  C_0^\infty(\tilde\Omega'_T)$ with $\theta\geq 0$.
If instead the left hand side is non-positive $u$ is said to be a \emph{weak subsolution}.

\subsection{Proof of Theorem \ref{thm1}}
Recall the function $\Gamma_{x_{n+1}}$ introduced in \eqref{aah}. Then to start the proof of Theorem \ref{T:quotients} we first observe, for ${x_{n+1}}>0$ fixed, that
extending $\Gamma_{x_{n+1}}$ by $0$ for $t<0$ we have
           \begin{eqnarray}\label{aa++}
     \int_{-\infty}^\infty\int_{\mathbb R^n}\Gamma_{x_{n+1}}(x,t)\, dxdt&=&\int_{0}^\infty\frac {1}{4^s\Gamma(-s)}{x_{n+1}}^{1-a}t^{-1-s}\exp(-|{x_{n+1}}|^2/(4t))\, dt\notag\\
     &=&\int_{0}^\infty\frac {1}{4^s\Gamma(-s)}t^{-1-s}\exp(-1/(4t))\, dt=1,
     \end{eqnarray}
     as we see by a change of variables. Furthermore,
      \begin{eqnarray}\label{aa++q}
      \Gamma_{x_{n+1}}(x,t)\to\delta_{(0,0)}(x,t)\mbox{ as }{x_{n+1}}\to 0,
      \end{eqnarray}
      where $\delta_{(0,0)}(x,t)$ is the Dirac delta at $(0,0)$,
      and hence
          \begin{eqnarray}\label{aa++q+}
      u(x,x_{n+1},t)\to f(x,t) \mbox{ as }{x_{n+1}}\to 0.
      \end{eqnarray}
      To continue we note, whenever $(x,t)\neq (0,0)$, ${x_{n+1}}>0$, that
            \begin{eqnarray}\label{aa++a}
\Delta\Gamma_{x_{n+1}}(x,t)&=&\Gamma_{x_{n+1}}\biggl (\frac {|x|^2}{(2t)^2}-\frac n {2t}\biggr ),\notag\\
\partial_t\Gamma_{x_{n+1}}(x,t)&=&\Gamma_{x_{n+1}}\biggl (-(n/2+1+s)t^{-1}+\frac {(|x|^2+{x_{n+1}}^2)}{(2t)^2}\biggr ).
\end{eqnarray}
Furthermore,
            \begin{eqnarray}\label{aa++b}
\partial_{x_{n+1}}\Gamma_{x_{n+1}}(x,t)&=&\Gamma_{x_{n+1}}\biggl (\frac {(1-a)}{x_{n+1}}-\frac {{x_{n+1}}}{2t}\biggr ),\notag\\
\partial^2_{x_{n+1}}\Gamma_{x_{n+1}}(x,t)&=&\Gamma_{x_{n+1}}\biggl (\frac {(1-a)}{x_{n+1}}-\frac {{x_{n+1}}}{2t}\biggr )^2+
\Gamma_{x_{n+1}}\biggl (\frac {(a-1)}{{x_{n+1}}^2}-\frac {1}{2t}\biggr ).
\end{eqnarray}
To be careful we note that
            \begin{eqnarray*}\label{aa++b+}
a{x_{n+1}}^{a-1}\partial_{x_{n+1}}\Gamma_{x_{n+1}}(x,t)&=&\Gamma_{x_{n+1}}\biggl ({a(1-a)}{x_{n+1}}^{a-2}-\frac{a{x_{n+1}}^a}{2t}\biggr ),\notag\\
{x_{n+1}}^a\partial^2_{x_{n+1}}\Gamma_{x_{n+1}}(x,t)&=&\Gamma_{x_{n+1}}\biggl ({(1-a)^2}{x_{n+1}}^{a-2}+\frac {{x_{n+1}}^{a+2}}{({2t})^{2}}-\frac{{x_{n+1}}^a(1-a)}t\biggr )\notag\\
&&+\Gamma_{x_{n+1}}\biggl ({(a-1)}{x_{n+1}}^{a-2}-\frac{{x_{n+1}}^a}{2t}\biggr ).
\end{eqnarray*}
Hence,
            \begin{eqnarray*}\label{aa++b+}
{x_{n+1}}^a\partial^2_{x_{n+1}}\Gamma_{x_{n+1}}(x,t)+a{x_{n+1}}^{a-1}\partial_{x_{n+1}}\Gamma_{x_{n+1}}(x,t)=\Gamma_{x_{n+1}}\biggl (\frac {{x_{n+1}}^{a+2}}{({2t})^{2}}+
\frac{(a-3){x_{n+1}}^a}{2t}\biggr ).
\end{eqnarray*}
Putting the calculations together we see that
            \begin{eqnarray*}\label{aa++b+a}
            {x_{n+1}}^a\partial_t\Gamma_{x_{n+1}}-\div({x_{n+1}}^a\nabla \Gamma_{x_{n+1}})
\end{eqnarray*}
equals
            \begin{eqnarray*}\label{aa++b+a}
         &&\Gamma_{x_{n+1}}\biggl (-(n/2+1+s)\frac {{x_{n+1}}^a}t+{x_{n+1}}^a\frac {(|x|^2+{x_{n+1}}^2)}{(2t)^2}\biggr )\notag\\
            &&-\Gamma_{x_{n+1}}\biggl (\frac {{x_{n+1}}^{a+2}}{({2t})^{2}}+
\frac{(a-3){x_{n+1}}^a}{2t}\biggr )\notag\\
&&-\Gamma_{x_{n+1}}{x_{n+1}}^a\biggl (\frac {|x|^2}{(2t)^2}-\frac n {2t}\biggr )\notag\\
&=&-\Gamma_{x_{n+1}}\frac{{x_{n+1}}^a}t\biggl (n/2+1+s+\frac{(a-3)}{2}-n/2\biggr )\notag\\
&=&-\Gamma_{x_{n+1}}\frac{{x_{n+1}}^a}t\biggl (n/2+1+s-1-s-n/2\biggr )=0.
\end{eqnarray*}
In particular, we can conclude that
\begin{eqnarray*}\label{aa++b+a}
            {x_{n+1}}^a\partial_tu(X,t)-\div({x_{n+1}}^a\nabla u(X,t))&=&0,\ \quad\quad (X,t)\in\mathbb R^{n+2}_+,\notag\\
            u(x,0,t)&=&f(x,t),\ (x,t)\in\mathbb R^{n+1}.
\end{eqnarray*}
We now consider the limit of $-{x_{n+1}}^a\partial_{x_{n+1}} u(X,t)$ as ${x_{n+1}}\to 0$. Indeed, we see that
        \begin{eqnarray}\label{aa++q+}
      -{x_{n+1}}^a\partial_{x_{n+1}} u(X,t)\biggl |_{x_{n+1}=0},
      \end{eqnarray}
      equals
       \begin{eqnarray}\label{aa++q+a}
      &&-\lim_{{x_{n+1}}\to 0} \frac {u(X,t)-u(x,0,t)}{{x_{n+1}}^{1-a}}\notag\\
            &=&\lim_{{x_{n+1}}\to 0}\int_{-\infty}^t\int_{\mathbb R^n}(f(x,t)-f(x',t')){x_{n+1}}^{a-1}\Gamma_{x_{n+1}}(x-x',t-t')\, dx'dt'\notag\\
            &=&\frac {1}{4^s\Gamma(-s)}\int_{-\infty}^t\int_{\mathbb R^n}\frac {(f(x,t)-f(x',y'))}{(t-t')^{1+s}}W(x-x',t-t')\, dx'dt'\notag\\
            &=&4^{-s}(\partial_t-\Delta)^sf(x,t).
      \end{eqnarray}
      This completes the proof of the first part of the theorem. To proceed we introduce
\begin{eqnarray}  \label{pol3aa+}
Q_r(X,t)=C_r(x,t)\times(x_{n+1}-r,x_{n+1}+r),
\end{eqnarray}
whenever $(X,t)=(x,x_{n+1},t)\in\mathbb R^{n+2}$ and $r>0$, and we let $\tilde u$ be as stated in the theorem. Let $\psi\in C_0^\infty(Q_r(\tilde x,0,\tilde t))$. We want to prove that
         \begin{eqnarray*}\label{aa++b+a}
         \int_{\mathbb R^{n+2}}|x_{n+1}|^a\bigl(-\tilde u(X,t)\partial_t\psi(X,t)+\nabla \tilde u(X,t)\cdot\nabla \psi(X,t)\bigr )\, dXdt=0.
\end{eqnarray*}
We note that
         \begin{eqnarray*}\label{aa++b+a}
         &&\int_{\mathbb R^{n+2}}|x_{n+1}|^a\bigl(-\tilde u(X,t)\partial_t\psi(X,t)+\nabla \tilde u(X,t)\cdot\nabla \psi(X,t)\bigr )\, dXdt\notag\\
         &=&\int_{Q_r(\tilde x,0,\tilde t)\setminus\{|x_{n+1}|<\epsilon\}}|x_{n+1}|^a\bigl(-\tilde u(X,t)\partial_t\psi(X,t)+\nabla \tilde u(X,t)\cdot\nabla \psi(X,t)\bigr )\, dXdt\notag\\
         &&+\int_{Q_r(\tilde x,0,\tilde t)\cap\{|x_{n+1}|<\epsilon\}}|x_{n+1}|^a\bigl(-\tilde u(X,t)\partial_t\psi(X,t)+\nabla \tilde u(X,t)\cdot\nabla \psi(X,t)\bigr )\, dXdt\notag\\
         &=:&I_1+I_2.
\end{eqnarray*}
By local integrability of $|x_{n+1}|^a|\tilde u(X,t)|^2$ and $|x_{n+1}|^a|\nabla \tilde u(X,t)|^2$ we have that $I_2\to 0$ as $\epsilon\to 0$. Furthermore, using the equation we see that
   \begin{eqnarray*}\label{aa++b+a}
         I_1&=& \int_{Q_r(\tilde x,0,\tilde t)\setminus\{|x_{n+1}|<\epsilon\}}\bigl (|x_{n+1}|^a\partial_t\tilde u(X,t)-\div(|x_{n+1}|^a\nabla \tilde u(X,t))\bigr )\psi(X,t)\, dXdt\notag\\
         &&-2\int_{Q_r(\tilde x,0,\tilde t)\setminus\{|x_{n+1}|=\epsilon\}}|\epsilon|^a\partial_{x_{n+1}}u(x,\epsilon,t)\partial_{x_{n+1}}\psi(x,\epsilon,t)\, dxdt\notag\\
         &=&-2\int_{Q_r(\tilde x,0,\tilde t)\setminus\{|x_{n+1}|=\epsilon\}}|\epsilon|^a\partial_{x_{n+1}}u(x,\epsilon,t)\partial_{x_{n+1}}\psi(x,\epsilon,t)\, dxdt.
\end{eqnarray*}
Hence also
   \begin{eqnarray*}\label{aa++b+a}
         I_1&=&-2\int_{Q_r(\tilde x,0,\tilde t)\setminus\{|x_{n+1}|=\epsilon\}}|\epsilon|^a\partial_{x_{n+1}}u(x,\epsilon,t)\partial_{x_{n+1}}\psi(x,\epsilon,t)\, dxdt\to 0
\end{eqnarray*}
as $\epsilon\to 0$ based on the assumption that       \begin{eqnarray*}\label{aa++q+aa}
      (\partial_t-\Delta)^sf(x,t)=0\mbox{ whenever }(x,t)\in C_r(\tilde x,\tilde t).
      \end{eqnarray*} This completes the proof of the second part of the theorem.

\section{Proof of Theorem \ref{T:quotients}}

Using Theorem \ref{thm1} we see that to prove Theorem \ref{T:quotients} it suffices to prove the following theorem.
\begin{theorem}\label{T:quotients.old}
Let $\mathcal{H}$ be as in  \eqref{lop}, assume \eqref{lop+} and \eqref{aabb}.
Let $\tilde\Omega\subset\mathbb R^{n+1}$ be a bounded
Lipschitz domain
with parameters $M$, $r_0$
and let $\tilde\Omega_T=\tilde\Omega\times (0,T)$ for some $T>0$.
Let $u, v$ be non-negative
solutions of $\mathcal{H}u=0$ in $\tilde\Omega_T$ vanishing continuously on $\tilde S_T$.
Let $\delta$, $0<\delta<r_0$, be a fixed constant.
Then $u/v$ is H{\"o}lder continuous
on the closure of $\tilde\Omega\times(\delta^2,T]$. Furthermore, let $(X_0,t_0)\in \tilde S_T$, $\delta^2\leq t_0$, and assume that $r<\delta/2$. Then there exist  $c=c(\mathcal{H},M,\diam(\tilde\Omega),T,\delta)$, $1\leq c<\infty$, and $\alpha=\alpha(\mathcal{H},M,\diam(\tilde\Omega),T,\delta)$, $\alpha\in (0,1)$, such that
\[  \biggl |\frac { u ( X,t ) }{ v ( X,t ) }-\frac { u ( Y,s ) }{ v ( Y,s ) }\biggr |
\leq c\biggl (\frac {|X-Y|+|t-s|^{1/2}}{r}\biggr )^\alpha\frac {u\bigl( A_r(X_0,t_0)\bigr) }{ v \bigl(A_r(X_0,t_0)\bigr) }\]
whenever $(X,t), (Y,s)\in \tilde\Omega_T\cap Q_{r/4}(X_0,t_0)$.
\end{theorem}
\begin{proof} As mentioned in the introduction, as $A=\{a_{ij}\}$ is independent of $t$ it is easily seen, using \cite{CS} and  \cite{CUR}, that Theorem \ref{T:quotients.old} follows from the by now classical arguments in \cite{FGS} once the continuous Dirichlet problem can be solved and once H{\"o}lder continuity up to the lateral boundary can be established. To solve the continuous Dirichlet problem one can use Perron's method. Indeed, for every $g\in C(\partial_p\Omega_T)$ there is a
unique Perron solution $u=u_g$ constructed in the usual way. To show that $u\in C(\overline{\Omega_T})$ one uses a barrier argument. In our case $A=\{a_{i,j}\}$ is independent of time and we can use the solvability of the Dirichlet problem in Lipschitz domains
for the degenerate elliptic operator
\[
\mathcal L=\sum_{i,j=1}^{n+1} \partial_{x_i}\left(a_{ij}\partial_{x_j}\right),
\]
see \cite{FKS}, to find a barrier at points $(X_0,t_0)\in \tilde S_T$. Indeed, let $\psi$ be the unique weak solution to the Dirichlet problem
$\mathcal L \psi =-1$ in $\tilde\Omega$ and $\psi(X)=|X-X_0|$ on $\partial \tilde \Omega$.
Then, using \cite{FKS} we have $\lim_{X\to X_0}\psi=0$ and $\lim_{X\to X'}\psi>0$
for all $X'\in \partial\Omega$, $X'\neq X_0$.
Let $\psi(X,t)=\psi(X)+(t_0-t)$, note that $\mathcal L \psi=-1=\partial_t \psi$ and thus that $\mathcal H \psi =0$. It follows that $\psi$ is a barrier at $(X_0,t_0)$
with respect to $\tilde \Omega\times(0,t_0)$. To establish  H{\"o}lder continuity up to $\tilde S_T$, assume that $u$ vanishes continuously on $\tilde S_T\cap Q_r(X_0,t_0)$ for some $(X_0,t_0)\in \tilde S_T$, $\sqrt{t_0}<r<r_0$. As  $\tilde \Omega$ is Lipschitz there exists a $\phi$ such that,
after a change of variables,
\[
\tilde\Omega_T \cap Q_r(X_0,t_0) = \{(x,x_{n+1},t)\in \R^{n+2} : x_{n+1}>\phi(x) \}\cap Q_r(X_0,t_0).
\]
Let $\rho(X,t)=\rho(x,x_{n+1},t)=(x,x_{n+1}+\phi(x),t)$
and note that there is a $c=c(M)$ such that
$\rho(\tilde Q_{r/c}(x_0,0,t_0))\subset  Q_r(X_0,t_0)$.
Denote $Q_{r/c}(x_0,0,t_0)$ by $Q^*$.
Let $\hat \lambda =\lambda\circ \phi$ and
$\hat u =u\circ \phi$. Then $\hat\lambda\in A_2$, with a constant depending only on $\Lambda$ and $M$,
and $\hat u$ is a solution to the equation
\[
\hat{\lambda} \partial_t u - \sum_{i,j=1}^{n+1} \partial_{x_i}(\hat{a}_{ij}\partial_{x_j} u)=0
\]
in $Q^*\cap \{x_{n+1}>0\}$, for a matrix valued function $\{\hat{a}_{ij}\}$ satisfying
\[
\hat\beta^{-1}\hat{\lambda} |\xi|^2
\le \sum_{i,j} \hat{a}_{ij} \xi_i\xi_j
\le \hat\beta \hat{\lambda} |\xi|^2,\ \xi\in \mathbb R^{n+1},
\]
in $Q^*\cap \{x_{n+1}>0\}$ and for some $\hat\beta=\hat\beta(n,\beta)$. As $u$ vanishes continuously on $Q_r(X_0,t_0)\cap\{x_{n+1}=\phi(x)\}$
it follows that $\hat u$ vanishes continuously on $Q^*\cap \{x_{n+1}=0\}$.
Reflecting the solution over $x_{n+1}=0$ in a standard manner, see \cite{Sa} and \cite{Is} for example, it follows that the extended function is a solution to a PDE of the same type. The H{\"o}lder regularity now follows from the corresponding interior
H{\"o}lder regularity of solutions, see \cite{CS} and \cite{Is}. We omit further details.
\end{proof}

    \end{document}